\documentclass[10pt,reqno]{amsart}
\usepackage{amsmath,amsfonts,amsthm,amssymb}
\usepackage[shortlabels]{enumitem}
\usepackage{tikz}
\usepackage{soul}
\usepackage{hyperref}

\newcommand{\R}{\mathbb{R}}
\newcommand{\N}{\mathbb{N}}

\newcommand{\supp}{\operatorname{supp}}

\newcommand{\bs}{\setminus}

\theoremstyle{plain}
\newtheorem{theorem}{Theorem}[section]
\newtheorem{lemma}[theorem]{Lemma}

\newtheorem{proposition}[theorem]{Proposition}

\theoremstyle{definition}
\newtheorem{remark}[theorem]{Remark}
\newtheorem{definition}[theorem]{Definition}
\newtheorem{example}[theorem]{Example}

\numberwithin{equation}{section}

\begin{document}
\title[A note on irreducibility for topical maps]{A note on irreducibility for topical maps}
\author[B. Lins]{Brian Lins}
\date{}
\address{Brian Lins, Hampden-Sydney College}
\email{blins@hsc.edu}

\begin{abstract}
Topical maps are a nonlinear generalization of nonnegative matrices acting on the interior of the standard cone $\R^n_{\ge 0}$. Several analogues of irreducibility have been defined for topical maps, and all are sufficient to guarantee the existence of entrywise positive eigenvectors.  In this note, we organize several of these notions, showing which conditions are stronger and when different types of irreducibility are equivalent.  We also consider how to computationally check the conditions. We show that certain irreducibility conditions can be expressed as Boolean satisfiability problems that can be checked using SAT solvers. This can be used to confirm the existence of entrywise positive eigenvectors when the dimension is large. 
\end{abstract}

\maketitle

The \emph{standard cone} in $\R^n$ is the set $\R^n_{\ge 0} = [0,\infty)^n$ of entrywise nonnegative vectors in $\R^n$. Its interior is $\R^n_{>0}$, the set of vectors with all positive entries. For $x, y \in [-\infty, \infty]^n$, we say that $x \ge y$ if $x_i \ge y_i$ for all $i \in [n]$ and $x \gg y$ if $x_i > y_i$ for all $i \in [n]$.  Let $D \subseteq [-\infty, \infty]^n$.  A function $f: D \rightarrow [-\infty, \infty]^n$ is \emph{order-preserving} if $f(x) \le f(y)$ for all $x, y \in D$ with $x \le y$.  We say that $f$ is \emph{homogeneous} if $f(\lambda x) = \lambda f(x)$ for all $x \in D$ and all positive real constants $\lambda$. A \emph{multiplicatively topical map} (shortened to \emph{m-topical} for convenience), is a function $f:\R^n_{\ge 0} \rightarrow \R^n_{\ge 0}$ that is continuous, order-preserving, homogeneous, and leaves the interior of the standard cone invariant, that is $f(\R^n_{>0}) \subseteq \R^n_{>0}$. 

In the literature, the term \emph{topical map} is usually reserved for \emph{additively topical maps} which are functions $T: \R^n \rightarrow \R^n$ that are order-preserving and additively homogeneous, that is $T(x + \lambda \mathbf{1}) = T(x) + \lambda \mathbf{1}$ for all $x \in \R^n$ and $\lambda \in \R$ where $\mathbf{1} \in \R^n$ is the vector with all entries equal to 1.  Note that if $T$ is topical and $\log$ and $\exp$ denote the entrywise natural logarithm and exponential functions, then the map  $f = \exp \circ T \circ \log$ is order-preserving and homogeneous on $\R^n_{>0}$. As such, it is known \cite[Corollary 2]{BurbanksSparrow99} (see also \cite[Corollary 4.6]{BurbanksNussbaumSparrow03} and \cite[Theorem 5.1.5]{LemmensNussbaum}) that $f$ extends uniquely to a continuous, order-preserving, and homogeneous map on all of $\R^n_{\ge 0}$, and therefore there is a one-to-one correspondence between additively and multiplicatively topical maps. 

Examples of topical maps include max-plus linear maps, min-max systems (\cite{Gunawardena04}, \cite{YaZh04}), Shapely operators for games \cite{AkianGaubertHochart20}, the Menon operator used in the Sinkhorn-Knopp algorithm \cite{BrualdiParterScheider66}, models in population biology \cite{Nussbaum89}, and certain nonlinear maps associated with nonnegative tensors \cite{ChPeZh08, FriedlandGaubertHan13, HuQi16, Lins23}. 
 
According to the Perron-Frobenius theorem, a nonnegative irreducible square matrix has a unique (up to scaling) eigenvector with all positive entries.  Several generalizations of irreducibility have been proposed for m-topical maps \cite{Morishima64, SchneiderTurner72, Oshime83, GaubertGunawardena04}.  Each is sufficient to guarantee the existence of eigenvectors in $\R^n_{>0}$. They do not typically guarantee uniqueness of the eigenvector, except in some special circumstances such as for real analytic m-topical maps \cite[Theorem 5.1]{Lins23}.  

The purpose of this note is twofold.  First we seek to organize the various generalizations of irreducibility and show how they relate. In section \ref{sec:types} we describe three main types of irreducibility for topical maps.  In section \ref{sec:almost} we describe some other combinatorial conditions for topical maps that also guarantee the existence of an eigenvector with positive entries. 

The second objective of this note is to introduce a method for checking irreducibility conditions using SAT solvers.  The irreducibility conditions we describe depend solely on the behavior of the maps on the boundary of the standard cone (and the inverted standard cone $(0,\infty]^n$). These conditions are combinatorial in nature, but the number of conditions to be checked can grow exponentially as the dimension increases.  Therefore these methods can become intractable in large dimensions.  However, in section \ref{sec:SAT} we describe how to restate some of the most general irreducibility conditions as Boolean satisfiability problems. This suggests that SAT solvers can be a powerful tool to certify that an m-topical map has an eigenvector with positive entries.  
We conclude in section \ref{sec:unique} by observing that the problem of checking whether an entrywise positive eigenvector is unique up to scaling can also be expressed as a Boolean satisfiability problem.

\section{Preliminaries} \label{sec:pre}

We begin with a quick review of some basic results from non-linear Perron-Frobenius theory. First, every m-topical map has an eigenvector in $\R^n_{\ge 0}$ corresponding to its cone spectral radius \cite[Corollary 5.4.2]{LemmensNussbaum}. 

\begin{proposition} \label{prop:KR}
An m-topical map $f$ on $\R^n_{\ge 0}$ has an eigenvector $x \in \R^n_{\ge 0}$ with eigenvalue equal to the cone spectral radius
$$r(f) := \lim_{k \rightarrow \infty} \|f^k(u)\|^{1/k}$$
where $u$ is any element of $\R^n_{>0}$. 
\end{proposition}

\emph{Hilbert's projective metric} on $\R^n_{>0}$ is defined by 
$$d_H(x,y) = \inf \{\log (\beta/\alpha) : \alpha x \le y \le \beta x \}.$$
Observe that $d_H(ax, by) = d_H(x,y)$ for all $a, b > 0$.  
An m-topical map $f:\R^n_{\ge 0} \rightarrow \R^n_{\ge 0}$ is always \emph{nonexpansive} with respect to Hilbert's projective metric, that is, for all $x, y \in \R^n_{>0}$, $d_H(f(x), f(y)) \le d_H(x,y)$ \cite[Corollary 2.1.4]{LemmensNussbaum}.  
Although m-topical maps always have entrywise nonnegative eigenvectors, they do not always have eigenvectors in the interior of the standard cone.  The following well known result gives a necessary and sufficient condition for the existence of an entrywise positive eigenvector \cite[Proposition 6.3.1]{LemmensNussbaum}

\begin{proposition} \label{prop:boundedorbit}
Let $f$ be an m-topical map on $\R^n_{\ge 0}$. There is an eigenvector $x \in \R^n_{>0}$ if and only if there exists a nonempty subset $W \subset \R^n_{>0}$ such that $W$ is bounded in Hilbert's projective metric and $f(W) \subseteq W$.
\end{proposition}


Let $f:\R^n_{\ge 0} \rightarrow \R^n_{\ge 0}$ be an m-topical map. A vector $x \in \R^n_{>0}$ is a \emph{super-eigenvector} of $f$ with super-eigenvalue $\beta > 0$ if $f(x) \le \beta x$.  It is a \emph{sub-eigenvector} with sub-eigenvalue $\alpha > 0$ if $\alpha x \le f(x)$.  For any $\alpha, \beta > 0$, we define the \emph{super-eigenspace} associated with $\beta$ to be
$$S^\beta(f) := \{x \in \R^n_{>0} : f(x) \le \beta x \}$$
and the \emph{sub-eigenspace} associated with $\alpha$ is
$$S_\alpha(f) := \{x \in \R^n_{>0} : \alpha x \le f(x) \}.$$
The intersection of a super and a sub-eigenspace is called a \emph{slice space}, and is denoted
$$S_\alpha^\beta(f) := \{x \in \R^n_{>0} : \alpha x \le f(x) \le \beta x\}.$$
Each of the sets $S_\alpha(f)$, $S^\beta(f)$, and $S_\alpha^\beta(f)$ is invariant under $f$.  Therefore, if any of these sets is nonempty and bounded in Hilbert's projective metric, then $f$ has an eigenvector in $\R^n_{>0}$ by Proposition \ref{prop:boundedorbit}. 

Note that all three of these sets are defined as subsets of the interior of the standard cone; however it is possible for an m-topical map to have other super and sub-eigenvectors on the boundary of $\R^n_{\ge 0}$.  A nontrivial sub-eigenvector in the boundary would be a non-zero vector $x \in \partial \R^n_{\ge 0}$ such that $f(x) \ge \alpha x$ for some $\alpha > 0$. Likewise, a nontrivial super-eigenvector in the boundary would be a non-zero $x \in \partial \R^n_{\ge 0}$ such that $f(x) \le \beta x$ for some $\beta < \infty$.  

The \emph{inverted standard cone} is the set $(0,\infty]^n$. The topology on $(0,\infty]^n$ comes from the order-topology on the extended real line. The following result is due to Burbanks and Sparrow \cite[Corollary 2]{BurbanksSparrow99} (see also \cite[Corollary 4.6]{BurbanksNussbaumSparrow03}).

\begin{proposition} \label{prop:extend}
Any m-topical map on $\R^n_{\ge 0}$ extends continuously to $(0,\infty]^n$.
\end{proposition}

Because $f$ extends continuously to $(0,\infty]^n$, we can also consider nontrivial super and sub-eigenvectors on the boundary of $(0, \infty]^n$. A vector $x \in (0,\infty]^n$ is a sub-eigenvector of an m-topical map $f$ if $x$ has at least one finite entry and $f(x) \ge \alpha x$ for some $\alpha > 0$. A vector $y \in (0,\infty]^n$ is a super-eigenvector if $y$ has at least one finite entry and $f(y) \le \beta y$ for some real $\beta > 0$.  For $x \in [0, \infty]^n$ we define its \emph{support} to be the set $\supp(x) = \{i \in [n] : 0 <  x_i < \infty\}$.

A map $f: \R^n_{\ge 0} \rightarrow \R^n_{\ge 0}$ is \emph{subadditive} if $f(x + y) \le f(x) + f(y)$ for all $x, y \in \R^n_{\ge 0}$.  If $f$ is homogeneous, then $f$ is subadditive if and only if $f$ is convex. Note that linear m-topical maps are subadditive, as are linear maps in the max-times algebra. 

If $f:\R^n_{\ge 0} \rightarrow \R^n_{\ge 0}$ is m-topical, then we say that $f$ is \emph{multiplicatively convex} (or just \emph{m-convex}) 
if $f(x^\lambda y^{1- \lambda}) \le f(x)^{\lambda} f(y)^{1- \lambda}$ for all $x, y \in \R^n_{>0}$ and $0 < \lambda < 1$, where the exponents and products are understood to be taken in each entry. Note that any sum or composition of m-convex m-topical maps is m-convex \cite[Lemma 4.1]{Lins23}.  A large class $\mathcal{M}_+$ of maps studied by Nussbaum in \cite{Nussbaum89} (see also \cite{LemmensNussbaum}) are m-convex m-topical maps.  These include a class of m-topical maps related to nonnegative tensors that have been widely studied \cite{ChPeZh08, FriedlandGaubertHan13, HuQi16, Lins23}.  

\begin{lemma}
If $f: \R^n_{\ge 0} \rightarrow \R^n_{\ge 0}$ is m-topical and subadditive, then $f$ is m-convex.  
\end{lemma} 

\begin{proof}
Since $f$ is homogeneous and subadditive, it is a convex function. Suppose that $x, y \in \R^n_{>0}$ and $z = x^{\lambda} y^{1-\lambda}$ for some $0< \lambda < 1$.  Since $f$ is convex, there is a matrix $A \in \R^{n \times n}$ that is a subdifferential for $f$ at $z$, that is, $f(z) = Az$ and $f(v) \ge Av$ for all $v \in \R^n_{>0}$.

Since $f$ is order-preserving, it follows that the entries of $A$ are nonnegative.  The linear transformation associated with $A$ maps $\R^n_{>0}$ into itself, so it is a multiplicatively convex map on $\R^n_{>0}$ (see e.g., \cite[Lemma 4.1]{Lins23}). Therefore 
$$f(z) = Az = A(x^\lambda y^{1-\lambda}) \le (Ax)^\lambda (Ay)^{1-\lambda} \le f(x)^\lambda f(y)^{1- \lambda}.$$
\end{proof}

\section{Types of Irreducibility} \label{sec:types}

Here we consider three different irreducibility conditions for m-topical maps. The (closed) \emph{faces} of $\R^n_{\ge 0}$ are the sets 
\begin{equation} \label{eq:face}
F_J := \{ x \in \R^n_{\ge 0} : x_i = 0 \text{ for all } i \notin J \}
\end{equation}
where $J \subseteq [n]$. A face $F_J$ is \emph{proper} if $J \neq [n]$, and it is \emph{nontrivial} if $J$ is nonempty.

\begin{definition} \label{def:facial}
An m-topical map is \emph{facially irreducible} if it does not leave any nontrivial proper closed faces of $\R^n_{\ge 0}$ invariant. 
\end{definition}

A second notion of irreducibility was proposed by Gaubert and Gunawardena \cite{GaubertGunawardena04}. For any m-topical map $f$ on $\R^n_{\ge 0}$, they define a directed graph $G(f)$ that generalizes the adjacency graph of a matrix.  For any subset $J \subseteq [n]$, let $e_J \in \R^n$ be the vector with entries
$$(e_J)_i = \begin{cases} 
1 & \text{ if } i \in J \\
0 & \text{ otherwise.} \end{cases}$$ 
The graph $G(f)$ has nodes $[n]$, and an arc from node $i$ to node $j$ if 
$$\lim_{t \rightarrow \infty} f(\mathbf{1} + t e_{\{j\}})_i = \infty.$$ 
\begin{definition} \label{def:graphical}
An m-topical map $f$ is \emph{graphically irreducible} if $G(f)$ is strongly connected.  
\end{definition}

Gaubert and Gunawardena \cite{GaubertGunawardena04} also suggest another irreducibility condition, which they refer to as indecomposability.  

\begin{definition} \label{def:indecomposable}
An m-topical map $f$ is \emph{decomposable} if there is a nonempty proper subset $J \subset [n]$ such that $$\lim_{t \rightarrow \infty} f(\mathbf{1} + te_J)_i < \infty$$
for all $i \in [n] \bs J$.  In other words, 
\begin{equation} \label{eq:omega}
\omega_J := \lim_{t \rightarrow \infty} \mathbf{1} + te_J \in (0,\infty]^n
\end{equation} 
is a super-eigenvector for $f$. The map $f$ is \emph{indecomposable} if it not decomposable. 
\end{definition}

Gaubert and Gunawardena proved that an m-topical map is indecomposable if and only if every super-eigenspace $S^\beta(f)$ is bounded in Hilbert's projective metric on $\R^n_{>0}$ \cite[Theorem 5]{GaubertGunawardena04}. 

\begin{theorem} \label{thm:connect1}
Let $f:\R^n_{\ge 0} \rightarrow \R^n_{\ge 0}$ be an m-topical map.  
\begin{enumerate}
\item \label{item:face=>indecomp} If $f$ is facially or graphically irreducible, then $f$ is indecomposable. 
\item \label{item:mconvex1} If $f$ is m-convex and indecomposable, then $f$ is graphically irreducible.  
\item \label{item:subadditive} If $f$ is subadditive and indecomposable, then it is facially irreducible. 
\end{enumerate}
\end{theorem}

\begin{proof}
Proof of \eqref{item:face=>indecomp}.  Let $J$ be any nonempty proper subset of $[n]$. If $f$ is facially irreducible, there must be an $i \in [n] \bs J$ such that $f(e_J)_i > 0$. Then 
$$\lim_{t \rightarrow \infty} f(\mathbf{1} + t e_J)_i \ge \lim_{t \rightarrow \infty} t f(e_J)_i = \infty.$$
So we conclude that $f$ is indecomposable.  

If $G(f)$ is strongly connected, then for any nonempty proper subset $J \subset [n]$, there must be an edge from some $i \notin J$ to some $j \in J$.  So 
$$\lim_{t \rightarrow \infty} f(\mathbf{1} + t e_J)_i \ge \lim_{t \rightarrow \infty} f(\mathbf{1} + t e_{\{j\}})_i = \infty.$$
Therefore $f$ is indecomposable. 

~\\
Proof of \eqref{item:mconvex1}. 
Suppose that $f$ is m-convex and indecomposable, but $G(f)$ is not strongly connected.  Let $I$ be a final class of $G(f)$ and let $J := [n] \bs I$.  Since $I$ is final, there are no edges from $i \in I$ to any $j \in J$.  Therefore 
$$\lim_{t \rightarrow \infty} f(\mathbf{1} + t e_{\{j\}})_i < \infty$$
for each $i \in I$ and $j \in J$.  
Observe that the entrywise geometric mean of the vectors $\{ \mathbf{1} + t e_{\{j\}} : j \in J\}$ is $\mathbf{1} + t^{1/|J|} e_J$. Since $f$ is m-convex, it follows that 
$$f(\mathbf{1} + t^{1/|J|} e_J)_i  \le \left( \prod_{j \in J} f(\mathbf{1} + t e_{\{j\}})_i \right)^{1/|J|}$$
for every $i \in I$. Therefore 
$$\lim_{t \rightarrow \infty} f(\mathbf{1} + t e_J)_i =\lim_{t \rightarrow \infty} f(\mathbf{1} + t^{1/|J|} e_J)_i  \le \left( \prod_{j \in J} \lim_{t \rightarrow \infty} f(\mathbf{1} + t e_{\{j\}})_i \right)^{1/|J|} < \infty.$$
Since this holds for all $i \notin J$, we have contradicted the assumption that $f$ is indecomposable, so we conclude that $G(f)$ must be strongly connected and therefore $f$ is graphically irreducible.  


~\\
Proof of \eqref{item:subadditive}.
Suppose that $f$ is subadditive and indecomposable.  Consider any nonempty proper subset $J \subset [n]$. 
Since $f$ is indecomposable, there must be at least one $i \in [n] \bs J$ such that 
$$\lim_{t \rightarrow \infty} f(\mathbf{1} + te_J)_i  = \infty.$$
Then by subadditivity, 
$$\lim_{t \rightarrow \infty} f(\mathbf{1})_i + tf(e_J)_i  = \infty.$$
So we conclude that $f(e_J)_i > 0$, and therefore the face $F_J$ defined by \eqref{eq:face} cannot be invariant.  This proves that $f$ is facially irreducible. 
\end{proof}

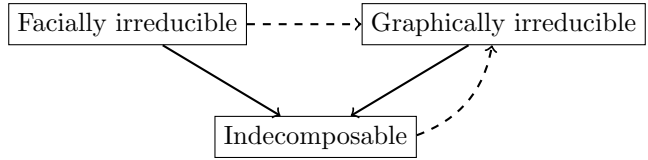
\begin{figure}[h] \label{fig:irreducibility}
\begin{center}
\begin{tikzpicture}
\path (0,0) node[draw,shape=rectangle,scale=1.0] (I1) {Facially irreducible};
\path (5,0) node[draw,shape=rectangle,scale=1.0] (I2) {Graphically irreducible};
\path (2.5,-1.5) node[draw,shape=rectangle,scale=1.0] (I3) {Indecomposable};

\draw[thick,dashed,->] (I1) to (I2);
\draw[thick,dashed,bend right,->] (I3) to (I2);
\draw[thick,->] (I1) to (I3);
\draw[thick,->] (I2) to (I3);

\end{tikzpicture}
\end{center}
\caption{Summary of Theorem \ref{thm:connect1}. Facial and graphical irreducibility both imply indecomposability for all m-topical maps. Graphical irreducibility is equivalent to indecomposability and weaker than facial irreducibility for m-convex m-topical maps (dashed edges). For subadditve m-topical maps, all three types of irreducibility are equivalent. }
\end{figure}

\begin{example}
Graphical irreducibility has been called weak irreducibility in some some contexts (see e.g., \cite{FriedlandGaubertHan13}).  Graphical irreducibility is weaker than facial irreducibility for m-convex maps, but not for m-topical maps in general.  For example, the following map is facially irreducible, but is not graphically irreducible.  
$$f(x) = \begin{bmatrix} x_2 + x_3 \\ x_1 + x_3 \\ (x_1^{-1} + x_2^{-1})^{-1} \end{bmatrix}.$$
\end{example}

\section{Other conditions like irreducibility} \label{sec:almost}

Two elements $x, y \in [0,\infty]^n$ are \emph{comparable} if there exist real $\alpha, \beta > 0$ such that $\alpha x \le y \le \beta x$.  Comparability is an equivalence relation on $\R^n_{\ge 0}$ and also on $(0,\infty]^n$. The equivalence classes are called the \emph{parts} of the (inverted) standard cone.  Each part of the standard cone has the form 
\begin{equation} \label{eq:part}
P_J = \{ x \in \R^n_{\ge 0} : x_i > 0 \text{ if and only if } i \in J \}
\end{equation}
where $J$ is a subset of $[n]$. If $x,y \in \R^n_{\ge 0}$ are comparable and $f: \R^n_{\ge 0} \rightarrow \R^n_{\ge 0}$ is order-preserving and homogeneous, then $f(x)$ and $f(y)$ are also comparable.  In particular, $f$ maps parts to parts. This suggest another definition, which is similar to, but weaker than, facial irreducibility. 
\begin{definition} \label{def:partial}
An m-topical map $f$ is \emph{partially irreducible} if $f$ has no invariant parts other than the trivial parts $P_\varnothing = \{0\}$ and $P_{[n]} = \R^n_{>0}$. 
\end{definition}
Unlike the previous notions, partial irreducibility does not correspond to irreducibility for nonnegative matrices.  For example, the linear transformation defined by the matrix 
$$A = \begin{bmatrix} 0 & 1 \\ 0 & 1 \end{bmatrix}$$ 
is partially irreducible even though $A$ is not an irreducible matrix.  It is immediately clear from the definition that partial irreducibility is implied by facial irreducibility.  It turns out that partial irreducibility is not implied by either graphical irreducibility or indecomposability.  

\begin{example}
Let $f:\R^2_{\ge 0} \rightarrow \R^2_{\ge 0}$ be defined by
$$f(x) = \begin{bmatrix} x_1 + \sqrt{x_1 x_2} \\ x_2 + \sqrt{x_1 x_2} \end{bmatrix}.$$
Then $f$ is an m-convex, m-topical map that is graphically irreducible and therefore indecomposable.  It is not partially or facially irreducible however, since the parts $P_{\{1\}}$ and $P_{\{2 \}}$ are both invariant.  
\end{example}

Akian, Gaubert, and Hochart gave a combinatorial necessary and sufficient condition for all slice spaces of an m-topical map to be bounded in Hilbert's projective metric in \cite{AkianGaubertHochart20}. For reasons that will become clear in Theorem \ref{thm:slice}, we refer to their condition as imperturbability. 

\begin{definition} \label{def:imperturb}
An m-topical map $f:\R^n_{\ge 0} \rightarrow \R^n_{\ge 0}$ is \emph{imperturbable}
if there are no nonempty subsets $I \subseteq J \subsetneq [n]$ such that $f(e_I)_i > 0$ for all $i \in I$ and 
$$\lim_{t \rightarrow \infty} f(\mathbf{1} + t e_J)_i < \infty$$
for all $i \notin J$. In other words, there is no $I \subseteq J$ such that $e_I$ is a sub-eigenvector and $\omega_J$ defined by \eqref{eq:omega} is a super-eigenvector for $f$. 
\end{definition}

The following characterization of imperturbable maps is \cite[Theorem 1.2]{AkianGaubertHochart20}. 
We include a proof because our notation and terminology is somewhat different than theirs. 
If $f, g$ are m-topical maps on $\R^n_{\ge 0}$, we say that $g$ is a \emph{Hilbert metric uniform perturbation} of $f$ if the Hilbert's projective metric distance between $f$ and $g$ is uniformly bounded on $\R^n_{>0}$.  

\begin{theorem} \label{thm:slice}
Let $f: \R^n_{\ge 0} \rightarrow \R^n_{\ge 0}$ be m-topical. The following are equivalent. 
\begin{enumerate}
\item \label{item:def} $f$ is imperturbable.
\item \label{item:slice} All slice spaces $S_\alpha^\beta(f)$ are bounded in Hilbert's projective metric. 
\item \label{item:D} $D \circ f$ has an eigenvector in $\R^n_{>0}$ for all diagonal matrices $D$ with positive diagonal entries. 
\item \label{item:perturb} Any m-topical map $g$ on $\R^n_{\ge 0}$ that is a Hilbert metric uniform perturbation of $f$ has an eigenvector in $\R^n_{>0}$. 
\end{enumerate}
\end{theorem}

\begin{proof}
The equivalence of conditions \eqref{item:slice}, \eqref{item:D}, and \eqref{item:perturb} is proved in \cite[Theorem 5.1]{Hochart19} in the additive setting.  The corresponding equivalence for multiplicatively topical maps follows since $\log \circ f \circ \exp$ is additively topical.  

\eqref{item:def} $\Rightarrow$ \eqref{item:slice}. Suppose by way of contradiction that $S_\alpha^\beta(f)$ is not bounded in Hilbert's projective metric.  Then we can choose a sequence $x^k \in S_\alpha^\beta$ that is not bounded in Hilbert's projective metric.  By passing to a subsequence, we can assume that the sequences $x^k/(\max_i x^k_i)$ and $x^k/(\min_i x^k_i)$ converge to $x \in [0,\infty)^n$ and $y \in (0,\infty]^n$ respectively.  By continuity, $f(x) \ge \alpha x$ and $f(y) \le \beta y$.  Since the sequence $x^k$ is not bounded in Hilbert's projective metric, the ratio of $\max_i x^k_i / \min_i x^k_i$ goes to infinity as $k \rightarrow \infty$.  Therefore $\supp(x) \cap \supp(y) = \varnothing$.  Let $I := \supp(x)$ and let $J := [n] \bs \supp(y)$ so that $I \subseteq J$. Since $e_I$ is comparable to $x$, we have $f(e_I)_i > 0$ for all $i \in I$. There is also a constant $c > 0$ such that $\mathbf{1} + t e_J \le c y$ for all $t > 0$. Therefore
$$\lim_{t \rightarrow \infty} f(\mathbf{1} + te_J)_i \le c f(y)_i < \infty.$$
This contradicts the assumption that $f$ is imperturbable.  


\eqref{item:D} $\Rightarrow$ \eqref{item:def}. Suppose that $f$ is not imperturbable. Then there exist nonempty subsets $I \subseteq J \subsetneq [n]$ such that $f(e_I)_i > 0$ for all $i \in I$ and $\lim_{t \rightarrow \infty } f(\mathbf{1} + te_J)_i < \infty$ for all $i \notin J$.  Let $\omega_J$ be defined as in \eqref{eq:omega}.
Since $f$ extends continuously to $(0,\infty]^n$ by Proposition \ref{prop:extend}, there is a finite constant $\beta > 0$ such that $f(\omega_J) \le \beta \omega_J$. There is also an $\alpha > 0$ such that $f(e_I) \ge \alpha e_I$.  
Observe that 
$$e_I \le \mathbf{1} \le \omega_J.$$
Let $D \in \R^{n \times n}$ be a diagonal matrix with positive diagonal entries $d_{ii}$ such that  $d_{ii} > 2 \alpha^{-1}$ when $i \in I$ and $d_{ii} < \tfrac{1}{2} \beta^{-1}$ when $i \notin J$. Then 
$$2^k e_I \le (D \circ f)^k(\mathbf{1}) \le 2^{-k} \omega_J$$
for all $k \in \N$.  This means that the forward orbit of $\mathbf{1}$ under the map $D \circ f$ is unbounded in Hilbert's projective metric. However, if $D \circ f$ had an eigenvector $u \in \R^n_{>0}$, then the Hilbert metric ball $W := \{x \in \R^n_{>0} : d_H(x, u) \le d_H(\mathbf{1}, u) \}$ would contain $(D\circ f)^k(\mathbf{1})$ for every $k \in \N$ by nonexpansiveness. 
Therefore $D \circ f$ cannot have an eigenvector in $\R^n_{>0}$.
\end{proof}

\begin{theorem} \label{thm:connect2}
Let $f:\R^n_{\ge 0} \rightarrow \R^n_{\ge 0}$ be m-topical. 
\begin{enumerate}
\item If $f$ is partially irreducible, then it is imperturbable.
\item If $f$ is subadditive and imperturbable, then it is partially irreducible.
\end{enumerate}
\end{theorem}
\begin{proof}
(1). If $f$ is partially irreducible, then so is $D \circ f$ for any diagonal matrix $D \in \R^{n \times n}$ with positive diagonal entries. Since $D \circ f$ is m-topical there is a constant $c > 0$ such that $(D \circ f)(\mathbf{1}) \ge c \mathbf{1}$.  Since $D \circ f$ is order-preserving and homogeneous, it follows that the spectral radius $r(D \circ f) \ge c > 0$.  This guarantees that $D \circ f$ has an eigenvector $x \in \R^n_{\ge 0}$ with eigenvalue $r(D \circ f)$ by Proposition \ref{prop:KR}. Since $D \circ f$ has no nontrivial invariant parts, $x$ must be in $R^n_{>0}$.  Since $D \circ f$ has an eigenvector in $\R^n_{>0}$ for every diagonal matrix $D$, Theorem \ref{thm:slice} implies that $f$ is imperturbable. 

(2). Suppose that $f$ subadditive but not partially irreducible.  Then there is a nonempty proper subset $J \subset [n]$ and constants $0 < \alpha < \beta$ such that 
$$\alpha e_J \le f(e_J) \le \beta e_J.$$
In particular, $f(e_J)_i > 0$ for all $i \in J$.  At the same time, 
$$\lim_{t \rightarrow \infty} f(\mathbf{1} + t e_J)_i \le \lim_{t \rightarrow \infty} f(\mathbf{1})_i + t f(e_J)_i = f(\mathbf{1})_i < \infty$$
for every $i \notin J$.  Therefore $f$ is not imperturbable.
\end{proof}

\begin{figure}[h] \label{fig:irreducibility2}
\begin{center}
\begin{tikzpicture}
\path (0,0) node[draw,shape=rectangle,scale=1.0] (I1) {Facially irreducible};
\path (5,0) node[draw,shape=rectangle,scale=1.0] (I2) {Graphically irreducible};
\path (2.5,-1.5) node[draw,shape=rectangle,scale=1.0] (I3) {Indecomposable};
\path (-2.5,-1.5) node[draw,shape=rectangle,scale=1.0] (I4) {Partially irreducible};
\path (0.0,-3.0) node[draw,shape=rectangle,scale=1.0] (I5) {Imperturbable};

\draw[thick,dashed,->] (I1) to (I2);
\draw[thick,->] (I1) to (I3);
\draw[thick,dashed,->,bend right] (I3) to (I2);
\draw[thick,->] (I2) to (I3);
\draw[thick, ->] (I1) to (I4);
\draw[thick, ->] (I3) to (I5);
\draw[thick,->] (I4) to (I5);

\end{tikzpicture}
\end{center}
\caption{Relationships between irreducibility conditions from Theorems \ref{thm:connect1} and \ref{thm:connect2}. Dashed arrows apply to m-convex maps. Partial irreducibility is equivalent to imperturbability for subadditive maps. }
\end{figure}
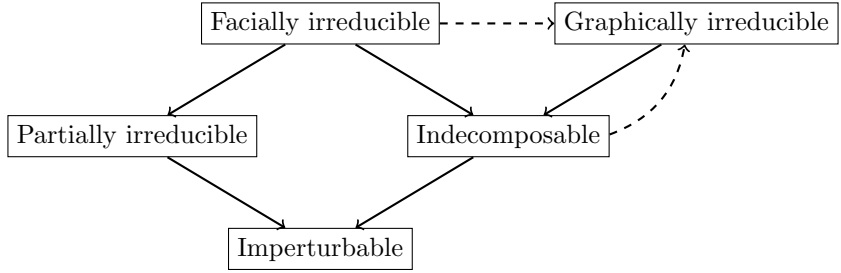



\section{Checking Irreducibility Conditions} \label{sec:SAT}

In this section we describe efficient ways to check the irreducibility conditions above. We begin by describing how irreducibility questions can be translated into Boolean satisfiability problems. 
A \emph{monotone Boolean function} is a function $g: \{0,1\}^n \rightarrow \{0, 1\}^n$ that is order-preserving. 
Each entry of a monotone Boolean function can be expressed as a Boolean formula using only the operations $\vee$ and $\wedge$, the entries $x_1, \ldots, x_n$ of the input, and the constants $0$ and $1$.   

\begin{definition} \label{def:sig1}
Let $f$ be an m-topical map on $\R^n_{\ge 0}$. The \emph{upper signature} of $f$ is the monotone Boolean function $\overline{f}$ on $\{0,1\}^n$ defined by 
$$\overline{f}(e_J)_i = 1 \text{ if and only if } \lim_{t \rightarrow \infty} f(\mathbf{1} + t e_J)_i = \infty.$$ 
The \emph{lower signature} of $f$ is the monotone Boolean function $\underline{f}$ on $\{0,1\}^n$ defined by
$$\underline{f}(e_J)_i = 1 \text{ if and only if } f(e_J)_i > 0.$$
\end{definition}

Note that both the upper and lower signature of a topical map have trivial fixed points $0$ and $\mathbf{1}$ in $\{0,1\}^n$.  This implies that both signatures can be expressed using only the operations $\vee$ and $\wedge$ without any constants.  

Combining Definition \ref{def:sig1} with Definitions \ref{def:facial}, \ref{def:indecomposable}, \ref{def:partial}, and \ref{def:imperturb}, we have the following alternative characterizations of facial irreducibility, indecomposability, partial irreducibility, and imperturbability. 

\begin{proposition}
Let $f$ be an m-topical map on $\R^n_{\ge 0}$ with upper and lower signatures $\overline{f}$ and $\underline{f}$.  
\begin{enumerate}
\item $f$ is facially irreducible if and only if there is no nontrivial $x \in \{0,1\}^n$ such that $\underline{f}(x) \le x$. 
\item $f$ is indecomposable if and only if there is no nontrivial $x \in \{0,1\}^n$ such that $\overline{f}(x) \le x$. 
\item $f$ is partially irreducible if and only if there is no nontrivial $x \in \{0,1\}^n$ such that $\underline{f}(x) = x$. 
\item $f$ is imperturbable if and only if there are no nontrivial $x, y \in \{0,1\}^n$ such that $x \le y$, $\underline{f}(x) \ge x$, and $\overline{f}(y) \le y$. 
\end{enumerate}
\end{proposition}


Nussbaum \cite{Nussbaum89} introduced a large class of m-topical maps that are constructed from averages. We will demonstrate how to find formulas for the upper and lower signatures of maps in an extended version of this class.  

For any $r \in \R$ and vector $\sigma \in \R^n_{\ge 0}$ with $\sum_i \sigma_i = 1$, the \emph{$(r,\sigma)$-average} of a vector $x \in \R^n_{\ge 0}$ is 
$$M_{r\sigma}(x) := \left(\sum_{i \in [n]} \sigma_i x_i^r \right)^{1/r} \text{ if } r \ne 0,$$
and
$$M_{r\sigma}(x) := \prod_{i \in [n]} x_i^{\sigma_i} \text{ if } r = 0.$$
Following \cite[Section 5.2]{AkianGaubertHochart20}, we also include two additional possibilities. When $r = \infty$, we let 
$$M_{\infty\sigma}(x) := \max_{i \in \supp(x)} x_i^{\sigma_i},$$
and when $r = -\infty$, 
$$M_{-\infty\sigma}(x) := \min_{i \in \supp(x)} x_i^{\sigma_i},$$
An m-topical function $f:\R^n_{\ge 0} \rightarrow \R^n_{\ge 0}$ is in \emph{class $\overline{M}$} if every entry of $f$ is a linear combination of $(r, \sigma)$-averages with $-\infty \le r \le \infty$.  If all of the $(r, \sigma)$-averages have $r \ge 0$ (respectively $r < 0$), then $f$ is in \emph{class $\overline{M}_+$} ($\overline{M}_-$). More generally \emph{class} $\overline{\mathcal{M}}$ is the smallest collection of m-topical maps that contains class $\overline{M}$ and is closed under addition and composition.  Likewise \emph{class} $\overline{\mathcal{M}}_+$ and $\overline{\mathcal{M}}_-$ are the smallest classes containing $\overline{M}_+$ and $\overline{M}_-$ respectively that are closed under addition and composition.  
It is also known (see e.g. \cite[Section 7.1]{Lins23}) that all maps in class $\overline{\mathcal{M}}_+$ are multiplicatively convex.  This includes a widely studied class of homogeneous maps that come from nonnegative tensors as well as max-algebra linear maps. 

If $f$ is in class $\overline{M}$, then it is simple to find formulas for its upper and lower signatures.  To get the upper signature $\overline{f}$, apply the following changes to the functions defining $f$: 
\begin{itemize}
\item Replace any $M_{r\sigma}$ with $M_{\infty \sigma}$ when $r \ge 0$ and with $M_{-\infty \sigma}$ when $r < 0$.  
\item Remove any constant coefficients.
\item Replace addition with the max operation $\vee$.  
\end{itemize}
The \emph{lower signature} $\underline{f}$ can be obtained using the same changes, except $M_{0 \sigma}$ is replaced by $M_{-\infty \sigma}$.
Note that a different signature function for maps in class $\overline{\mathcal{M}}$ was introduced in \cite[Section 5.2]{AkianGaubertHochart20}.  Their construction is similar to ours, but it is not generally a Boolean function. 
The following lemma shows how to find upper and lower signatures for maps in class $\overline{\mathcal{M}}$ that are built from maps in $\overline{M}$ using composition and linear combinations.  
\begin{lemma}
Let $f$ and $g$ be m-topical maps on $\R^n_{\ge 0}$. If $\overline{f}, \overline{g}$ are the upper signatures and $\underline{f}, \underline{g}$ are the lower signatures for $f$ and $g$ respectively, then
\begin{enumerate}
\item \label{item:comp} $\overline{f} \circ \overline{g}$ is the upper signature for $f \circ g$. 
\item \label{item:comp2} $\underline{f} \circ \underline{g}$ is the lower signature for $f \circ g$. 
\item \label{item:sum} $\overline{f} \vee \overline{g}$ is the upper signature for $a f + b g$ when $a, b > 0$.  
\item \label{item:sum2} $\underline{f} \vee \underline{g}$ is the lower signature for $a f + b g$ when $a, b > 0$.  
\end{enumerate}
\end{lemma}

\begin{proof}
Recall by Proposition \ref{prop:extend} that for any $J \subseteq [n]$ and any m-topical map $f$ on $\R^n_{\ge 0}$, 
$$\lim_{t \rightarrow \infty} f(\mathbf{1} + te_J) = f(\omega_J)$$
where $\omega_J$ is given by \eqref{eq:omega}.

Proof of \eqref{item:comp}. Let $J \subseteq [n]$.  Let $I = \supp(\overline{g}(e_J))$, so that $e_I = \overline{g}(e_J)$.  By definition, $g(\omega_J)_i$ is finite if and only if $i \notin I$.  Therefore there exist real constants $\alpha, \beta > 0$ such that $\alpha \omega_I \le g(\omega_J) \le \beta \omega_I$.  Then $\alpha f(\omega_I) \le f(g(\omega_J)) \le \beta f(\omega_I)$.  Therefore $f(g(\omega_J))_i = \infty$ if and only if $\overline{f}(e_I)_i = 1$ which is equivalent to $\overline{f}(\overline{g}(e_J))_i = 1$.  

Proof of \eqref{item:comp2}. Let $J \subseteq [n]$.  Let $I = \supp(g(e_J))$. Observe that $\underline{g}(e_J) = e_I$ and $g(e_J)$ is comparable to $e_I$.  Therefore $f(g(e_J))$ and $\underline{f}(\underline{g}(e_J))$ are comparable.  Therefore $f(g(e_J))_i > 0$ if and only if $\underline{f}(\underline{g}(e_J))_i = 1$.  


Proof of \eqref{item:sum}. 
For any $J \subseteq [n]$, $ af(\omega_J)_i + bg(\omega_J)_i  = \infty$ if and only if $f(\omega_J)_i = \infty$  or  $g(\omega_J)_i = \infty.$ The latter condition is equivalent to $\overline{f}(e_J)_i \vee \overline{g}(e_J)_i$ which proves the assertion.

Proof of \eqref{item:sum}. It is clear that $a f(e_J)_i + b g(e_J)_i > 0$ if and only if $f(e_J)_i > 0$ or $g(e_J)_i > 0$ which in turn holds if and only if $(\underline{f} \vee \underline{g})(e_J)_i = 1$.
\end{proof}

\begin{remark}
If we have a formula for the upper signature $\overline{f}$ of an m-topical map $f$ on $\R^n_{\ge 0}$, then we can translate the conditions for $f$ to be indecomposable into a Boolean satisfiability problem.  The map $f$ is decomposable if and only if there is a nontrivial $x \in \{0,1\}^n$ such that $\overline{f}(x) \le x$.  So $f$ is decomposable if and only if we can satisfy the following Boolean conditions.  

First, $x$ must be nontrivial which means we must satisfy the following condition
$$(x_1 \vee x_2 \vee \ldots \vee x_n) \wedge (\neg x_1 \vee \neg x_2 \vee \ldots \vee \neg x_n).$$
Second, we need $\overline{f}(x) \le x$, which corresponds to
$$(x_1 \vee \neg \overline{f}(x)_1) \wedge (x_2 \vee \neg \overline{f}(x)_2) \wedge \ldots \wedge (x_n \vee \neg \overline{f}(x)_n).$$
The first condition is already in conjunctive normal form.  The second condition can be converted to conjunctive normal form using the Tseitin transformation.  This may result in a longer Boolean formula with more variables, but the length of the resulting formula is linear in the length of the original.  Furthermore, the Tseitin transformation is built into many SAT-solvers such as Z3.  
\end{remark}

%
%

\begin{remark}
Checking whether an m-topical map $f$ is imperturbable with an SAT-solver requires converting the following collection of conditions to conjunctive normal form.  If $f$ is not imperturbable, then we should be able to find $x, y\in \{0,1\}^n$ such that the sets $I := \supp(x)$ and $J := \supp(y)$ are a counterexample to Definition \ref{def:imperturb}. So $x$ and $y$ should satisfy the following four Boolean conditions. First we need $x$ to be nonzero and $y \neq \mathbf{1}$.
$$(x_1 \vee x_2 \vee \ldots \vee x_n) \wedge (\neg y_1 \vee \neg y_2 \vee \ldots \vee \neg y_n)$$
We also need $\supp(x) \subseteq \supp(y)$.  
$$(\neg x_1 \vee y_1) \wedge (\neg x_2 \vee y_2) \wedge \ldots \wedge (\neg x_n \vee y_n).$$
We need $\underline{f}(x) \ge x$ so 
$$(\neg x_1 \vee \underline{f}(x)_1) \wedge (\neg x_2 \vee \underline{f}(x)_2) \wedge \ldots \wedge (\neg x_n \vee \underline{f}(x)_n).$$
Finally we need $\overline{f}(y) \le y$.
$$(y_1 \vee \neg \overline{f}(y)_1) \wedge (y_2 \vee \neg \overline{f}(y)_2) \wedge \ldots \wedge (y_n \vee \neg \overline{f}(y)_n).$$
Observe that the first two conditions are already in conjunctive normal form, while the latter two are typically not. We can convert the last two Boolean expressions to conjunctive normal form using the Tseitin transformation.  Then the conjunction of all four conditions can be given to an SAT-solver to check.  
\end{remark}

\begin{remark}
It is considerably easier to check whether an m-convex m-topical maps is indecomposable or imperturbable. As noted in Theorem \ref{thm:connect1}\eqref{item:mconvex1} m-convex m-topical maps are indecomposable if and only if they are graphically irreducible. Therefore checking whether an m-convex m-topical map is indecomposable can be done by finding the strongly connected components of $G(f)$.  
These can be efficiently computed using Tarjan's algorithm, which runs in $O(n^2)$ time for m-topical maps on $\R^n_{\ge 0}$. 

If the graph $G(f)$ is not strongly connected, it can still help determine whether or not $f$ is imperturbable.  The following is \cite[Corollary 4.4]{AkianGaubertHochart20}, restated for the multiplicatively topical rather than additively topical setting. 

\begin{theorem}
If $f$ is an m-convex m-topical map on $\R^n$, then $f$ is imperturbable if and only if $G(f)$ has a unique final class $I \subseteq [n]$ and there exists $k > 0$ such that $f^k(e_{I^c}) = 0$.  
\end{theorem}
\end{remark}

\section{Unique eigenvectors} \label{sec:unique}

For nonnegative matrices, irreducibility guarantees not only the existence, but also the uniqueness (up to scaling) of an eigenvector with all positive entries.  It was observed in \cite[Theorem 5.1]{Lins23} that if $f$ is an m-topical map that is real analytic in $\R^n_{>0}$ and the set of eigenvectors of $f$ in $\R^n_{>0}$ is nonempty and bounded in Hilbert's projective metric, as is the case for imperturbable maps, then the eigenvector must be unique up to scaling.  The importance of real analyticity was also noted by Oshime who showed that any facially irreducible m-topical map which is real analytic on $\R^n_{>0}$ has a unique eigenvector in $\R^n_{>0}$ \cite[Proposition 3.9]{Oshime92}.  

Confirming that an eigenvector is unique can be more challenging for maps which are not real analytic, Morishima \cite{Morishima64} introduced an irreducibility condition for m-topical maps which guarantees not only the existence, but also the uniqueness of an eigenvector in $\R^n_{>0}$. Note that Morishima used the name indecomposability for this condition, however it is not to indecomposability as defined by Definition \ref{def:indecomposable}.   Later Oshime \cite{Oshime83} described a weakening of Morishima's condition called non-sectionality which also guarantees the existence and uniqueness of an entrywise positive eigenvector. Both Morishima and Oshime's conditions make use of the following condition. 

\begin{definition}
Let $f$ be an m-topical map on $\R^n_{\ge 0}$ and let $x \in \R^n_{>0}$.  Then $f$ satisfies \emph{condition (M)} at $x$ if for every nonempty proper subset $J \subset [n]$, there exists $i \notin J$ such that
$$f(x+t e_J)_i > f(x)_i.$$
\end{definition}

Morishima's condition is equivalent to facial irreducibility combined with the assumption that condition (M) is satisfied at every $x \in \R^n_{>0}$. Oshime's non-sectionality condition is equivalent to indecomposability combined with condition (M) at every $x \in \R^n_{>0}$.  

Requiring condition (M) at every $x \in \R^n_{>0}$ is unnecessary if you know the exact value of one eigenvector $u \in \R^n_{>0}$. In that case, confirming that condition (M) holds at $u$ is sufficient to guarantee that $u$ is the only eigenvector up to scaling.  

\begin{lemma} \label{lem:condM}
Let $f$ be an m-topical map on $\R^n$ with an eigenvector $u \in \R^n_{>0}$.  If $f$ satisfies condition (M) at $u$, then $u$ is the only eigenvector of $f$ in $\R^n_{>0}$ up to scaling.  
\end{lemma}

\begin{proof}
We may assume by scaling $f$ that $r(f) = 1$. Suppose that $v \in \R^n_{>0}$ is a second eigenvector for $f$ that is not a scalar multiple of $v$. Both $u$ and $v$ have eigenvalues equal to 1, since by Proposition \ref{prop:KR}
$$r(f) = \lim_{k \rightarrow \infty} \|f^k(u)\|^{1/k} = \lim_{k \rightarrow \infty} \|f^k(v)\|^{1/k} = 1.$$ 
We may assume by scaling $v$ that $v \ge u$ and $v_i = u_i$ for some $i \in [n]$. Let $I = \{i \in [n] : v_i = u_i \}$ and $J = [n] \bs I$. There exists $t > 0$ small enough so that $u + te_J \le v.$ 
But then 
$$u_i = f(u)_i < f(u + te_J)_i \le f(v)_i = v_i$$
for every $i \notin J$. This contradicts the assumption that $u_i = v_i$ for all $i \in I$, so we conclude that there cannot be two linearly independent eigenvectors in $\R^n_{>0}$. 
\end{proof}

A weakening of condition (M) leads to necessary and sufficient conditions for an eigenvector in $\R^n_{>0}$ to be unique.

\begin{definition}
Let $f$ be an m-topical map on $\R^n_{\ge 0}$ and let $x \in \R^n_{>0}$.  Then $f$ satisfies \emph{condition (N)} at $x$ if for every pair of nonempty disjoint sets $I, J \subset [n]$, either there exists $j \in J$ such that
$$f(x+t e_{J^c})_j > f(x)_j$$
for all $t > 0$ or there exists $i \in I$ such that
$$f(x-t e_{I^c})_i < f(x)_i$$ 
for all $t > 0$ sufficiently small.
\end{definition}

The following is \cite[Theorem 4.8]{AkianGaubertHochart20}.
\begin{theorem} \label{thm:condN}
Let $f$ be an m-topical map on $\R^n$ with an eigenvector $u \in \R^n_{>0}$.  Then $u$ is the unique eigenvector of $f$ in $\R^n_{>0}$ up to scaling if and only if $f$ satisfies condition (N) at $u$. 
\end{theorem}

To determine whether either condition (M) or (N) holds, it helps to define local upper and lower signature functions.  If $f$ is m-topical on $\R^n_{\ge 0}$ and $u \in \R^n_{>0}$, then the \emph{upper local signature} of $f$ at $u$ is the monotone Boolean function $\overline{f}_u$ on $\{0,1\}^n$ defined by
$$\overline{f}_u(e_J)_i = 1 \text{ if and only if } f(u + te_J)_i > f(u)_i \text{ for all } t > 0.$$  
The \emph{lower local signature} of $f$ at $u$ is the monotone Boolean function $\underline{f}_u$ on $\{0,1\}^n$ defined by
$$\underline{f}_u(e_J)_i = 1 \text{ if and only if } f(u - te_J)_i < f(u)_i \text{ for all } t > 0 \text{ sufficiently small}.$$  

It is easy to find upper and lower local signatures for maps in $\overline{M}$. If $f$ is an m-topical map on $\R^n_{\ge 0}$ in class $\overline{M}$, then the upper local signature of $f$ at $u$ is determined by applying these changes to the formula for each entry of $f$. 
\begin{itemize}
\item For each $M_{\pm \infty \sigma}$, choose $\tau, \mu \in \R^n_{\ge 0}$ such that 
$$\supp(\tau) = \left\{i \in \supp(\sigma) : u_i = \max_{j \in \supp(\sigma)} u_j \right\}$$
and
$$\supp(\mu) = \left\{i \in \supp(\sigma) : u_i = \min_{j \in \supp(\sigma)} u_j \right\}.$$
\item Replace each $M_{\infty \sigma}$ with $M_{\infty \tau}$.
\item Replace $M_{-\infty \sigma}$ with $M_{-\infty \mu}$. 
\item Remove any constant coefficients.  
\item Replace any $M_{r\sigma}$ with $M_{\infty \sigma}$ when $r \in \R$.  
\item Replace addition with the max operation $\vee$.  
\end{itemize}
The lower local signature $\underline{f}_u$ can be found similarly, except for $M_{\pm \infty \sigma}$ which are dealt with as follows:
\begin{itemize}
\item Replace each $M_{\infty \sigma}$ with $M_{-\infty \tau}$.  
\item Replace $M_{-\infty \sigma}$ with $M_{\infty \mu}$.
\end{itemize}

The following lemma describes how to find the upper and lower local signatures for maps in $\overline{\mathcal{M}}$ that are compositions and linear combinations of maps in $\overline{M}$.  

\begin{lemma}
Let $f, g$ be m-topical maps on $\R^n_{\ge 0}$.  Let $u \in \R^n_{>0}$ and $v = g(u)$.  Then 
\begin{enumerate}
\item The upper local signature of $f \circ g$ at $u$ is $\overline{f}_v \circ \overline{g}_u$ and the lower local signature of $f \circ g$ at $u$ is $\underline{f}_v \circ \underline{g}_u$.  
\item For any $a, b > 0$, the upper local signature of $af + bg$ is $\overline{f}_u \vee \overline{g}_u$ and the lower local signature of $af + bg$ is $\underline{f}_u \vee \underline{g}_u$. 
\end{enumerate}
\end{lemma}

\begin{proof}
Consider $f(g(u + te_J))$.  Let $I$ be the subset of $[n]$ such that $e_I = \overline{g}_u(e_J)$.  Then for all $t > 0$, there exists $s > 0$ such that $g(u+te_J) \ge g(u) + se_I$.  So 
$$f(g(u+te_J))_i \ge f(g(u) + s e_I)_i > f(g(u))_i$$
if and only if $\overline{f}_v(\overline{g}_u(e_J))_i  = 1$.  Thus $\overline{(f \circ g)}_u = \overline{f}_v \circ \overline{g}_u$. The proof that $\underline{(f \circ g)}_u = \underline{f}_v \circ \overline{g}_u$ is essentially the same. 

Now for any $a, b > 0$, observe that 
\begin{align*}
af(u + te_J) + bg(u+te_J) &- af(u) - bg(u)  \\ 
&= a(f(u+te_J) - f(u)) + b (g(u+te_J) - g(u)).
\end{align*}
Therefore 
$$af(u + te_J)_i + bg(u+te_J)_i > af(u)_i + bg(u)_i$$ 
if and only if  
$$f(u+te_J)_i > f(u)_i \text{ or } g(u+te_J)_i > g(u).$$
This shows that $\overline{(af+bg)}_u = \overline{f}_u \vee \overline{g}_u$.  The proof for the lower local signature is the same. 
\end{proof}

\begin{remark}
Determining whether an m-topical map $f$ on $\R^n_{\ge 0}$ satisfies condition (N) at $u \in \R^n_{>0}$ can be translated into a Boolean satisfiability problem using the upper and lower local signatures for $f$.  Indeed, $f$ satisfies condition (N) at $u$ if and only if 
for every pair of nonempty disjoint sets $I, J \subset [n]$, either $\overline{f}_u(e_{J^c})_j = 1$ for some $j \in J$ or $\underline{f}_u(e_{I^c})_i = 1$ for some $i \in I$. 

In the Boolean satisfiability problem, we are looking for $x, y \in \{0,1\}^n$ whose supports correspond to the sets $I^c$ and $J^c$.  If $u$ is not unique, then by Theorem \ref{thm:condN} we should be able to find $x, y$ with the following properties.  
First, the supports of $x$ and $y$ are both proper subsets of $[n]$. 
$$(\neg x_1 \vee \neg x_2 \vee \ldots \vee \neg x_n) \wedge (\neg y_1 \vee \neg y_2 \vee \ldots \vee \neg y_n).$$
Second, the supports of $x$ and $y$ cover $[n]$.
$$(x_1 \vee y_1) \wedge (x_2 \vee y_2) \wedge \ldots (x_n \vee y_n).$$
Then we need $\underline{f}_u(e_{I^c})_i = 0$ for all $i \in I$ and $\overline{f}_u(e_{J^c})_j = 0$ for all $j \in J$. The former can be expressed as 
$$(x_1 \vee \neg \underline{f}_u(x)_1) \wedge (x_2 \vee \neg \underline{f}_u(x)_2) \wedge \ldots \wedge (x_n \vee \neg \underline{f}_u(x)_n).$$
The last condition is 
$$(y_1 \vee \neg \overline{f}_u(y)_1) \wedge (y_2 \vee \neg \overline{f}_u(y)_2) \wedge \ldots \wedge (y_n \vee \neg \overline{f}_u(y)_n).$$
\end{remark}

\begin{remark}
If $f$ is m-convex and m-topical on $\R^n_{\ge 0}$ and $u \in \R^n_{>0}$, then \cite[Corollary 4.9]{AkianGaubertHochart20} gives a graph algorithm that can check whether or not $u$ is unique in $O(n^2)$ time.  
\end{remark}

\bibliography{DW2}
\bibliographystyle{plain}

\end{document}